\documentclass{amsart}
\usepackage{amsfonts,amssymb}

\newcommand{\e}{\varepsilon}
\newcommand{\U}{\mathcal U}
\newcommand{\V}{\mathcal V}
\newcommand{\dist}{\mathrm{dist}\,}
\newcommand{\diam}{\mathrm{diam}\,}
\newcommand{\IN}{\mathbb N}
\newcommand{\IZ}{\mathbb Z}
\newcommand{\IQ}{\mathbb Q}
\newcommand{\IR}{\mathbb R}
\newcommand{\Ra}{\Rightarrow}
\newcommand{\w}{\omega}
\newcommand{\lev}{\mathrm{lev}}
\newcommand{\Deg}{\mathrm{Deg}}
\newcommand{\mesh}{\mathrm{mesh}\,}
\newcommand{\C}{\mathcal C}
\newcommand{\A}{\mathcal A}

\newcommand{\suc}{\mathrm{pred}}

\newcommand{\Ent}{\mathrm{Ent}}
\newcommand{\ent}{\mathrm{ent}}
\newcommand{\id}{\mathrm{id}}
\newcommand{\next}{\mathrm{next}}

\newtheorem{theorem}{Theorem}
\newtheorem{corollary}{Corollary}
\newtheorem{lemma}{Lemma}
\newtheorem{proposition}{Proposition}
\newtheorem{problem}{Problem}

\theoremstyle{definition}
\newtheorem{definition}{Definition}

\title{The coarse classification of homogeneous ultra-metric spaces}
\author{Taras Banakh, Ihor Zarichnyy}
\address{Instytut Matematyki, Akademia \' Swi\c etokrzyska w Kielcach (Poland),\newline Department of Mathematics, Ivan Franko National University of Lviv (Ukraine)}
\email{tbanakh@yahoo.com}
\email{ihor.zarichnyj@gmail.com}
\begin{document}

\begin{abstract} We prove that two homogeneous ultra-metric spaces $X,Y$ are coarsely equivalent if and only if $\Ent^\sharp(X)=\Ent^\sharp(Y)$ where $\Ent^\sharp(X)$ is the so-called sharp entropy of $X$. This classification implies that each  homogeneous proper ultra-metric space is  coarsely equivalent to the anti-Cantor set $2^{<\w}$. For the proof of these results we develop a technique of towers which can have an independent interest.
\end{abstract}

\maketitle

\tableofcontents

\section*{Introduction}
In this paper we classify homogeneous ultra-metric spaces up to the coarse equivalence.

Let us recall some necessary definitions. We say that a metric space $(X,d)$ is
\begin{itemize}
\item {\em homogeneous} if for any two points $x,y\in X$ there is an isometrical bijection $f:X\to X$ with $f(x)=y$;
\item {\em proper} if $X$ is unbounded but for every $x_0\in X$ and $r\in[0,+\infty)$ the closed $r$-ball $B_r(x_0)=\{x\in X:d(x,x_0)\le r\}$ centered at $x_0$ is compact;
\item an {\em ultra-metric space} if $d(x,y)\le\max\{d(x,z),d(z,y)\}$ for every points $x,y,z\in X$.
\end{itemize}

The basic example of a homogeneous proper ultra-metric space 
is the space $$2^{<\w}=\{(x_i)_{i\in\w}\in 2^\w:\exists m\in\w\;\forall i\ge m\;\; x_i=0\}$$endowed with the ultrametric 
$$D(\vec x,\vec y)=
\max_{n\in\w} 2^n |x_n-y_n|,$$
where $\vec x=(x_n)_{n\in\w}$ and $\vec y=(y_n)_{n\in\w}$ are two points of $2^{<\w}$. Here $2=\{0,1\}$ and more generally, $\alpha=\{\beta:\beta<\alpha\}$ for any ordinal $\alpha$.

The ultra-metric space $2^{<\w}$, called the {\em anti-Cantor} set, is an asymptotic counterpart of the Cantor cube $2^\w$ endowed with the  ultrametric $$d(\vec x,\vec y)=\max_{n\in\w}2^{-n}|x_n-y_n|$$

By analogy, for every set $A$ with $|A|>1$ we can consider the countable product $(A^\w,d)$ and its asymptotic counterpart $(A^{<\w},D)$. According to the classical 
Brouwer theorem for each finite set $A$ with $|A|>1$ the countable  product $A^\w$ is (uniformly) homeomorphic to the Cantor cube $2^\w$.

The problem if the Brouwer theorem has an asymptotic counterpart has been circulated among asymptologists (see \cite[\S5]{BDHM}) and was communicated to the authors by Ihor Protasov. To answer this question we first need to recall the notion of the coarse equivalence, which relies on the notion of a bornologous map. By definition, a map $f:X\to Y$ between metric spaces is {\em bornologous} if for every $\e\in\IR$ there is $\delta\in\IR$ such that for each points $x,x'\in X$ with $\dist(x,x')\le\e$ we get $\dist(f(x),f(x'))\le\delta$.

\begin{definition} We say that two metric spaces $X,Y$ are 
\begin{itemize}
\item {\em bijectively asymorphic} if there is a bornologous bijective map $f:X\to Y$ with bornologous inverse $f^{-1}$;
\item {\em coarsely equivalent} if there are bornologous maps $f:X\to Y$ and $g:Y\to X$ such that $\dist(g\circ f,\id_X)<\infty$ and $\dist(f\circ g,\id_Y)<\infty$.
\end{itemize}
\end{definition}
In Section~\ref{sa} we shall give several equivalent definitions of the coarse equivalence.

It is known  that for two finite sets $A,B$ the metric spaces $A^{<\w}$ and $B^{<\w}$ are bijectively asymorphic if and only if $|A|$ and $|B|$ have the same prime divisors, see \cite[10.6]{PB}, \cite[p.57]{PZ} or \cite[5.5]{BDHM}. In particular, $2^{<\w}$ and $3^{<\w}$ are not bijectively asymorphic. In light of this result, it is natural to ask if $2^{<\w}$ and $3^{<\w}$ are coarse equivalent, see \cite[\S5]{BDHM}. The positive answer to this question can be  easily derived from the homogeneity of $2^{<\w}$ and $3^{<\w}$ and the following theorem (that is a particular case of a more general  Theorem~\ref{ast} below). 

\begin{theorem}\label{t1} Any homogeneous proper ultra-metric space is coarsely equivalent to the anti-Cantor set $2^{<\w}$.
\end{theorem}

According to \cite[2.42]{Roe}, any two coarsely equivalent proper metric spaces $X,Y$ have homeomorphic Higson coronas $\nu X$, $\nu Y$. Combining this fact with Theorem~\ref{t1}, we get 

\begin{corollary}\label{c2} The Higson coronas $\nu X,\nu Y$ of any two homogeneous  proper ultra-metric spaces $X,Y$ are homeomorphic.
\end{corollary}

Theorem~\ref{t1} follows from a more general result detecting ultra-metric spaces 
coarsely equivalent to the Cantor set with helf of cardinal invariants called small and large entropies. Given a subset $B$ of a metric space $X$ and a real number $\e$ we define the {\em $\e$-entropy} $\Ent_\e(B)$ of $B$ as the smallest cardinality $|N|$ of an $\e$-net $N\subset B$ (the latter means that for each point $x\in B$ there is a point $y\in N$ with $\dist(x,y)<\e$). 
For $\e,\delta\in[0,\infty)$ let 
$$
\Ent^\delta_\e(X)=\sup_{x\in X}\Ent_\e(B_\delta(x))\;\;\mbox{ and }\;\;
\ent^\delta_\e(X)=\min_{x\in X}\Ent_\e(B_\delta(x))$$
where $B_\delta(x)=\{y\in X:\dist(x,y)\le\delta\}$ stands for the closed $\delta$-ball centered at $x$. A metric space $X$ is defined to have {\em bounded geometry} if there is $\e\in\IR$ such that $\Ent_\e^\delta(X)<\aleph_0$ for all $\delta\in\IR$.
For such spaces we have the following theorem implying 
Theorem~\ref{t1}.

\begin{theorem}\label{ast} A proper ultra-metric space $X$ is coarsely equivalent to the anti-Cantor set provided there is an increasing unbounded number sequence $\vec r=\{r_n\}_{n\in\IN}$ such that
$$\prod_{n\in\IN} \frac{\Ent_{r_n}^{r_{n+1}}(X)}{\ent_{r_n}^{r_{n+1}}(X)}<+\infty.$$
\end{theorem}

Theorem~\ref{t1} is the principal ingredient in the coarse classification of homogeneous
ultra-metric spaces. Such spaces are classified with help of a cardinal invariant called the {\em sharp entropy}. To define this cardinal invariant, for a metric space $X$ and a real number $\e$ let 
$$\Ent^\sharp_\e(X)=\sup_{\delta<\infty}\big(\Ent^\delta_\e(X)\big)^+\;\;\mbox{ and }\;\;
\ent^\sharp_\e(X)=\sup_{\delta<\infty}\big(\ent^\delta_\e(X)\big)^+$$be the {\em large} and {\em small} $\e$-{\em entropies} of $X$ (here by $\kappa^+$ we denote the successor cardinal to a cardinal $\kappa$). The cardinal numbers
$$\Ent^\sharp(X)=\min_{\e<\infty}\Ent^\sharp_\e(X)\;\;\mbox{ and }\;\;
\ent^\sharp(X)=\min_{\e<\infty}\ent^\sharp_\e(X)$$are called  the {\em large} and {\em small sharp entropies} of $X$, respectively.

It is clear that $\ent^\sharp(X)\le\Ent^\sharp(X)$ for any metric space $X$. If $X$ is homogeneous, then we have the equality $\ent^\sharp(X)=\Ent^\sharp(X)$ (because $\Ent_\e(B_\delta(x))=\Ent_\e^\delta(B_\delta(y))$ for all $\e,\delta$ and $x,y\in X$).

It follows that $\Ent^\sharp(X)\le\aleph_0$ if and only if there is $\e>0$ such that 
$\Ent_\e^\delta(X)<\aleph_0$ for all $\delta\in\IR$, which means that $X$ has bounded geometry. 

Observe that the sharp entropy distinguishes between the anti-Cantor set $2^{<\w}$ and the anti-Baire space $\IN^{<\w}$ because $\Ent^\sharp(2^{<\w})=\aleph_0$ while $\Ent^\sharp(\IN^{<\w})=\aleph_1$. 

The following classification theorem (implying Theorem~\ref{t1}) is one of the main result of this paper.

\begin{theorem}\label{class} Two homogeneous ultra-metric spaces are coarsely equivalent if and only if $\Ent^\sharp(X)=\Ent^\sharp(Y)$.
\end{theorem}

The following proposition completes Theorem~\ref{class} and presents some elementary properties of the sharp entropies.

\begin{proposition}\label{ent} \begin{enumerate}
\item If a metric space $X$ is coarsely equivalent to a subspace of a metric space $Y$, then $\Ent^\sharp(X)\le\Ent^\sharp(Y)$.
\item If two metric spaces $X,Y$ are coarsely equivalent, then $\Ent^\sharp(X)=\Ent^\sharp(Y)$ and $\ent^\sharp(X)=\ent^\sharp(Y)$.
\item An ultra-metric space $X$ is coarsely equivalent to a subspace of an ultra-metric space $Y$ provided $\Ent^\sharp(X)\le\ent^\sharp(Y)$.
\item For a cardinal number $\kappa$ there is a non-empty (proper homogeneous ultra-) metric space $X$ with $\Ent^\sharp(X)=\kappa$ if and only if either $\kappa=2$ or $\kappa$ is an infinite successor cardinal, or $\kappa$ is a limit cardinal of countable cofinality.
\end{enumerate}
\end{proposition}

The third item of the preceding proposition generalizes a result of A.Dranishnikov and M.Zarichnyi \cite{DZ} who proved that each ultra-metric space $X$ of bounded geometry is coarsely equivalent to a subspace of the anti-Cantor set $2^{<\w}$.
\smallskip

In fact, the above results apply not only to (homogeneous) ultra-metric spaces but, more generally to asymptotically zero-dimensional (homogeneous) metric spaces because any such a space is bijectively asymorphic to a (homogeneous) ultra-metric  space, see  Proposition~\ref{hom}.

\section{Characterizing the coarse equivalence}\label{sa}

In this section we show that various natural ways of defining morphisms in Asymptology\footnote{The term ``Asymptology'' was introduced by I.Protasov in \cite{PZ} for naming the theory studying large scale properties of metric spaces (or more general objects like {\em balleans} of I.~Protasov \cite{PZ}, \cite{PB} or {\em coarse structures} of J.~Roe \cite{Roe}).} lead to the same notion of coarse equivalence. Besides the original approach of J.~Roe based on the notion of a coarse map, we discuss an alternative approach based on the notion of a multi-map.

By a {\em multi-map} $\Phi:X\Ra Y$ between two sets $X,Y$ we understand any subset $\Phi\subset X\times Y$. For a subset $A\subset X$ by $\Phi(A)=\{y\in Y:\exists a\in A\mbox{ with }(a,y)\in\Phi\}$ we denote the image of $A$ under the multi-map $\Phi$. The inverse $\Phi^{-1}:Y\Ra X$ to the multi-map $\Phi$ is the subset $\Phi^{-1}=\{(y,x)\in Y\times X: (x,y)\in\Phi\}\subset Y\times X$. For two multi-maps $\Phi:X\Ra Y$ and $\Psi:Y\Ra Z$ we define their composition $\Psi\circ\Phi:X\Ra Z$ as usual:
$$\Psi\circ\Phi=\{(x,z)\in X\times Z:\exists y\in Y\mbox{ such that $(x,y)\in \Phi$ and $(y,z)\in\Psi$}\}.$$

 A multi-map $\Phi$ is called {\em surjective} if $\Phi(X)=Y$ and {\em bijective} if $\Phi\subset X\times Y$ coincides with the graph of a bijective (single-valued) function.

A multi-map $\Phi:X\Ra Y$ between metric spaces $X$ and $Y$ is called 
\begin{itemize}
\item {\em bornologous} if for any $\e>0$ there is $\delta>0$ such that for any subset $A\subset X$ with $\diam(A)<\e$ the image $B=\Phi(A)$ has diameter $\diam (B)<\delta$;
\item an {\em asymorphism} if both $\Phi$ and $\Phi^{-1}$ are surjective bornologous multi-maps;
\item an {\em asymorphic embedding} if both $\Phi$ and $\Phi^{-1}$ are bornologous multi-maps and $\Phi^{-1}$ is surjective.
\end{itemize}

It is clear that the composition of two surjective (bornologous) multi-maps is  surjective (and bornologous). Consequently, the composition of asymorphisms is an asymorphism.

\begin{definition}  We shall say that two metric spaces $X,Y$ are  ({\em bijectively}) {\em asymorphic}\footnote{In \cite{PZ} bijective asymorphisms are called asymorphisms while asymorphisms are referred to as quasi-asymorphisms. However we suggest to change the terminology shifting the attention to asymorphisms (in our sense) as a central concept of the  Asymptology.} and will denote this by $X\sim Y$ if there is a (bijective) asymorphism $\Phi:X\Ra Y$. 
\end{definition}

A subset $L$ of a metric space $X$ is called {\em large} if $O_r(L)=X$ for some $r\in\IR$, where $O_r(L)=\{x\in X:\dist(x,L)<\e\}$ stands for the open $r$-neighborhood of the set $L$ in $X$.

The following  characterization is the main (and unique) result of this section. 

\begin{proposition}\label{ascors} For metric spaces $X,Y$ the following assertions are equivalent:
\begin{enumerate}
\item $X$ and $Y$ are asymorphic;
\item $X$ and $Y$ are coarsely equivalent;
\item the spaces $X,Y$ contain bijectively asymorphic large subspaces $X'\subset X$ and $Y'\subset Y$;
\item there are two bornologous maps $f:X\to Y$, $g:Y\to X$ whose inverses $f^{-1}:Y\Ra X$ and $g^{-1}:X\Ra Y$ are bornologous multi-maps and $$\max\{\dist(g\circ f,\id_X),\dist(f\circ g,\id_Y)\}<\infty.$$
\end{enumerate}
\end{proposition}

\begin{proof} To prove the equivalence of the items (1)--(4), it suffices to establish the implications $(1)\Ra(4)\Ra(2)\Ra(3)\Ra(1)$.
\smallskip

$(1)\Ra(4)$ Assuming that $X$ and $Y$ are asymorphic, fix a surjective bornologous multi-map $\Phi:X\Ra Y$ with surjective bornologous inverse $\Phi^{-1}:Y\Ra X$. Since the multi-map $\Phi^{-1}$ is surjective, for every $x\in X$ there is a point $f(x)\in Y$ with $x\in\Phi^{-1}(f(x))$, which is equivalent to $f(x)\in\Phi(x)$.
It follows from the bornologity of $\Phi$ that the map $f:X\to Y$ is bornologous. Since $f^{-1}(y)\subset \Phi^{-1}(y)$ for all $y\in Y$, the bornologous property of $\Phi^{-1}$ implies that property for the multi-map $f^{-1}:Y\Ra X$.

By the same reason, the surjectivity of the multi-map $\Phi$ implies the existence of a map $g:Y\to X$ such that $g(y)\in\Phi^{-1}(y)$ for all $y\in Y$. The bornologity of $\Phi$ and $\Phi^{-1}$ implies that $g:Y\to X$ and $g^{-1}:X\Ra Y$ are bornologous. 

Since the composition $\Phi^{-1}\circ\Phi:X\Ra X$ is bornologous, there is a constant $C<\infty$ such that $\diam \Phi^{-1}\circ \Phi(x)\le C$. Observing that $\{x,g\circ f(x)\}\subset\Phi^{-1}\circ \Phi(x)$ we see that $\dist(g\circ f,\id_X)\le C<\infty$. By the same reason, $\dist(f\circ g,\id_Y)<\infty$.
\smallskip

The implication $(4)\Ra(2)$ trivially follows from the definition of the coarse equivalence given in the Introduction.
\smallskip

$(2)\Ra(3)$ Assume that there are two bornologous maps $f:X\to Y$, $g:Y\to X$ with $\dist(g\circ f,\id_X)\le R$ and $\dist(f\circ g,\id_Y)\le R$ for some real number $R$.
It follows that $O_R(f(X))=Y$ and hence the set $Y'=f(X)$ is large in $Y$. Choose any subset $X'\subset X$ making the restriction $h=f|X':X'\to Y'$ bijective. The bornologous property of $f$ implies that the bijective map $h:X'\to Y'$ is bornologous. 

Let us show that the inverse map $h^{-1}:Y'\to X'$ is bornologous.  Given arbitrary $\e<\infty$, use  the bornologity of the map $g:Y\to X$ to find a number $\delta<\infty$ such that $\diam g(C)<\delta$ for every set $C\subset Y$ with $\diam(C)\le\e$. 
Now take any points $y,y'\in Y'$ with $\dist(y,y')\le \e$ and let $x=h^{-1}(y)$ and $x'=h^{-1}(y')$. We claim that $\dist(x,x')\le\delta+2R$. Indeed, the choice of $\delta$ guarantees that $\dist(g(y),g(y'))\le\delta$. Since $\dist(g\circ f,\id_X)\le R$, we conclude that 
$$
\begin{aligned}
\dist(x,x')&\le\dist(x,g\circ f(x))+\dist(g\circ f(x),g\circ f(x'))+\dist(g\circ f(x'),x')\le\\
&\le R+\dist(g(y),g(y'))+R\le\delta+2R.
\end{aligned}
$$

Finally, let us show that the set $X'$ is large in $X$. Given any point $x\in X$, find a point $x'\in X'$ with $f(x)=f(x')$. Then $\dist(x,x')\le \dist(x,g\circ f(x))+\dist(g\circ f(x'),x')\le 2R$ and consequently, $O_{2R}(X')=X$.
\smallskip

$(3)\Ra(1)$ Assume that the spaces $X,Y$ contain bijectively asymorphic large subspaces $X'\subset X$ and $Y'\subset Y$. Let $f:X'\to Y'$ be a bijective asymorphism.
Find $R\in\IR$ such that $O_R(X')=X$ and $O_R(Y')=Y$. Take any maps $\varphi:X\to X'$ and $\psi:Y\to Y'$ with $\dist(\varphi,\id_X)\le R$ and $\dist(\psi,\id_Y)\le R$. It is easy to see that $\varphi$ and $\psi$ are asymorphisms and then the composition $\psi^{-1}\circ f\circ \varphi:X\Ra Y$ is a required asymorphism between $X$ and $Y$.
\end{proof}

\section{Towers}

The results stated in the Introduction are proved by induction on partially ordered sets called towers. Towers are order antipodes of trees but on the other hand, seen as graphs, the towers are trees in the graph-theoretic sense (i.e., are connected graphs without circuits). We recall that a partially ordered set $T$ is a {\em tree} if $T$ has the smallest element and for every point $x\in T$ the lower cone ${\downarrow}x$ is well-ordered. By the {\em lower cone} (resp. {\em upper cone}) of a point $x$ of a partially ordered set $T$ we understand  the set ${\downarrow}x=\{y\in T:y\le x\}$ (resp. ${\uparrow}x=\{y\in T:y\ge x\}$). A subset $A\subset T$ will be called a {\em lower} (resp. {\em upper}) {\em set} if ${\downarrow}a\subset A$ (resp. ${\uparrow}a\subset A$) for all $a\in A$. A partially ordered set $T$ is {\em well-founded} if each subset $A\subset T$ has a minimal element $a\in A$. The minimality of $a$ means that each point $a'\in A$ with $a'\le a$ is equal to $a$.
By $\min T$ we shall denote the set of all minimal elements of $T$.

Now we define the principal technical concept of this paper.

\begin{definition}\label{d1} A partially ordered set $T$ is called a {\em tower} if
\begin{enumerate}
\item $T$ is well-founded;
\item any two elements $x,y\in T$ have the smallest upper bound $\sup(x,y)$ in $T$;
\item for any $x\in T$ the upper cone ${\uparrow}x$ is linearly ordered;
\item for any point $a\in T$ there is a finite number $n=\lev_T(a)$ such that for every minimal element $x\in{\downarrow}a$ of $T$ the order interval $[x,a]={\uparrow}x\cap{\downarrow}a$ has cardinality $\big|[x,y]\big|=n$.
\end{enumerate}
\end{definition}
The function $\lev_T:T\to\IN$, $\lev_T:a\mapsto \lev_T(a)$, from the last item is called the {\em level function}. If the tower $T$ is clear from the context, then we  omit the subscript $T$ and write $\lev(a)$ instead of $\lev_T(a)$. One can observe that $\lev_T=1+\mathrm{rank}_T$ where $\mathrm{rank}_T$ is the usual rank function of the well-founded set $T$, see \cite[Appendix B]{Ke}.

The level function $\lev_T:T\to\IN$ divides $T$ into the levels $L_i=\lev_T^{-1}(i)$, $i\in\IN$. The 1-st level $L_1=\min T$ will be called {\em the base} of $T$ and will be denoted by $[T]$. The number $h(T)=\sup\,\{n\in\IN:L_n\ne\emptyset\}$ is called the {\em height} of the tower $T$. A tower $T$ is {\em unbounded} if it has infinite height.   The following model of the famous Eiffel tower is just an example of a tower of height 7.

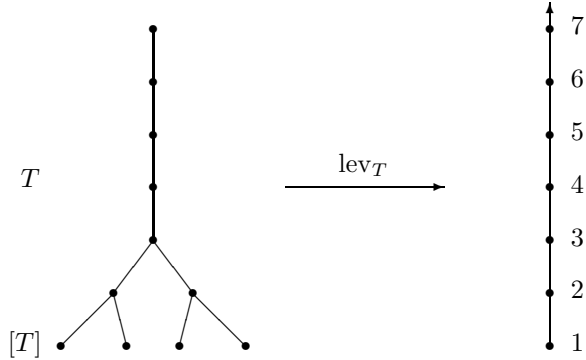
\begin{picture}(100,145)(-100,-10)

\put(-50,60){$T$}
\put(-55,-2){$[T]$}
\put(0,120){\circle*{3}}
\put(0,120){\line(0,-1){20}}
\put(0,100){\circle*{3}}
\put(0,100){\line(0,-1){20}}
\put(0,80){\circle*{3}}
\put(0,80){\line(0,-1){20}}
\put(0,60){\circle*{3}}
\put(0,60){\line(0,-1){20}}
\put(0,40){\circle*{3}}
\put(0,40){\line(-3,-4){15}}
\put(0,40){\line(3,-4){15}}
\put(-15,20){\circle*{3}}
\put(-15,20){\line(-1,-1){20}}
\put(-35,0){\circle*{3}}
\put(15,20){\line(1,-1){20}}
\put(15,20){\circle*{3}}
\put(-15,20){\line(1,-4){5}}
\put(15,20){\line(-1,-4){5}}
\put(-10,0){\circle*{3}}
\put(10,0){\circle*{3}}
\put(35,0){\circle*{3}}

\put(50,60){\vector(1,0){60}}
\put(70,66){$\lev_T$}

\put(150,120){\vector(0,1){10}}
\put(150,120){\circle*{3}}
\put(158,118){7}
\put(150,120){\line(0,-1){20}}
\put(150,100){\circle*{3}}
\put(158,98){6}
\put(150,100){\line(0,-1){20}}
\put(150,80){\circle*{3}}
\put(158,78){5}
\put(150,80){\line(0,-1){20}}
\put(150,60){\circle*{3}}
\put(158,58){4}
\put(150,60){\line(0,-1){20}}
\put(150,40){\circle*{3}}
\put(158,38){3}
\put(150,40){\line(0,-1){20}}
\put(150,20){\circle*{3}}
\put(158,18){2}
\put(150,20){\line(0,-1){20}}
\put(150,0){\circle*{3}}
\put(158,-2){1}

\end{picture}

In fact, towers of finite height are not interesting: they are trees in the reverse partial order. Because of that we shall assume that all towers are unbounded.

Each tower carries a canonic {\em path metric} $d_T$ defined by the formula 
$$d_T(x,y)=2\cdot\lev_T\big(\sup(x,y)\big)-\big(\lev_T(x)+\lev_T(y)\big)\mbox{ for $x,y\in T$}.$$
The path metric $d_T$ restricted to the base $[T]$ of $T$ is an ultrametric.
In the sequel talking about metric properties of towers we shall always refer to the path metric. 

A subset $S$ of an tower $T$ is called a {\em subtower} if $S$ is an tower in the induced partial order. For every tower $T$ and an increasing number sequence $\vec k=(k_n)_{n\in\w}$ the subset
$$
T(\vec k)=\{x\in T:\lev(x)\in\{k_n\}_{n\in\w}\}
$$is a subtower of $T$, called the {\em level subtower} of $T$ generated by the sequence $\vec k$, or briefly the {\em level $\vec k$-subtower} of $T$. 

It is easy to see that each unbounded subtower $S$ of a tower $T$ is {\em cofinal}  in $T$ in the sense 
that for every $t\in T$ there is $s\in S$ with $t\le s$. Given a cofinal subset $S\subset T$ consider the map $\next_S:T\to S$ assigning to each $x\in T$ the smallest point $y\in S$ with $y\ge x$ (such a smallest point exists because the upper set ${\uparrow}x$ is well-ordered). It is easy to see that $\next_S([T])\subset[S]$.
The following proposition trivially follows from the definitions. 
 
\begin{proposition}\label{next} Let $T$ be an tower and $S=T(\vec k)$ be a level subtower of $T$. Then the map $\next_S:[T]\to [S]$ is an asymorphism.
\end{proposition}
 
For every point $x\in T$ of a tower $T$ and a number $i\le \lev(x)$ let $\suc_i(x)=L_i\cap{\downarrow}x$ be the set of predecessors of $x$ in the $i$-th generation and $\deg_i(x)=|\suc_i(x)|$. For $i=\lev(x)-1$ the set $\suc_{i}(x)$ is called the set of parents of $x$ and is denoted by $\suc(x)$. The cardinality $|\suc(x)|$ is called the {\em degree} of $x$ and is denoted by $\deg(x)$. Thus $\deg(x)=\deg_{\lev(x)-1}(x)$. 

For an integer numbers $k\le n$ let  
$$\deg_k^n(T)=\min\{\deg_k(x):x\in L_{n}\}\mbox{ and }\Deg_k^n(T)=\sup\{\deg_k(x):x\in L_n\}.$$
We shall write $\deg_n(T)$ and $\Deg_n(T)$ instead of $\deg^{n+1}_n(T)$ and $\Deg^{n+1}_{n}(T)$, respectively.

The small and large entropies of the boundary $[T]$ of a tower $T$ can be easily calculated via the degrees $\deg_i^j(T)$ and $\Deg_i^j(T)$ of $T$.

\begin{proposition}\label{p7} For any tower $T$ we have
\begin{enumerate}
\item $\ent_{2i}^{2j}([T])=\deg^{j+1}_{i+1}(T)$ and $\Ent_{2i}^{2j}([T])=\Deg_{i+1}^{j+1}(T))^+$ for all $i\le j$;
\item $\ent_\sharp([T])=\min\limits_{i\in\IN}\sup\limits_{j>i}\,(\deg_i^j(T))^+$ and $\Ent^\sharp([T])=\min\limits_{i\in\IN}\sup\limits_{j>i}\,(\Deg_i^j(T))^+.$
\end{enumerate}
\end{proposition}

This proposition can be easily derived from the definition of the path metric on the boundary $[T]$ of $T$ and the definition of the small and large sharp entropies of $[T]$. 

In order to prove a tower counterpart of Proposition~\ref{ent}(3) we need a definition.

An injecive (resp. bijective) map $\varphi:T_1\to T_2$ will be called a {\em tower embedding} (resp. {\em a tower isomorphism}) if 
\begin{itemize}
\item $\varphi$ is {\em monotone} in the sense that $x\le y$ in $T_1$ implies $\varphi(x)\le\varphi(y)$ in $T_2$ and
\item {\em level-preserving}, which means that $\lev_{T_2}(\varphi(x))=\lev_{T_1}(x)$ for all $x\in T_1$.
\end{itemize}

This definition combined with the definition of the path metric of a tower implies

\begin{proposition} For each  tower embedding (isomorphism) $\varphi:T_1\to T_2$ the restriction $\varphi|[T_1]:[T_1]\to[T_2]$ is an isometric embedding (bijection).
\end{proposition}

Now we give conditions of towers $T_1,T_2$ guarantees the existence of a tower embedding (isomorphism) $T_1\to T_2$.

\begin{proposition}\label{p5a} For two towers $T_1,T_2$ there is a tower embedding (isomorphism) $\varphi:T_1\to T_2$  provided $\Deg_k(T_1)\le\deg_k(T_2)$ (and $\Deg_k(T_2)\le\deg_k(T_1)$) for all $k\in\IN$.
\end{proposition}

\begin{proof} Assume that $\Deg_k(T_1)\le\deg_k(T_2)$ \ $\big($and $\Deg_k(T_2)\le\deg_k(T_1)$~$\big)$ \ for all $k\in\IN$. We shall need the following

\begin{lemma}\label{l0} For any two points $u\in T_1$ and $v\in T_2$ with $\lev(u)=\lev(v)$ there is a tower embedding (isomorphism) $\varphi:{\downarrow}u\to{\downarrow}v$. Moreover, if for some $u_0\in\suc(u)$ and $v_0\in\suc(v)$ we are given with a tower embedding (isomorphism) $\varphi_0:{\downarrow}u_0\to {\downarrow}v_0$, then the map $\varphi$ can be chosen so that $\varphi|{\downarrow}u_0=\varphi_0$.
\end{lemma}

\begin{proof} The proof is by induction of the level $\lev(u)=\lev(v)$. If this level is 1, then there is nothing to construct: just put $\varphi:\{u\}\to\{v\}$ be the constant map. 

Now assume that the lemma has been proved for all $u,v$ with $\lev(u)=\lev(v)<n$. 

Take any points $u\in T_1$ and $v\in T_2$ with $\lev(u)=\lev(v)=n$. Consider the sets $\suc(u)$ and $\suc(v)$. Since  $\Deg_{n-1}(T_1)\le\deg_{n-1}(T_2)$, we conclude that $|\suc(u)|\le|\suc(v)|$ and thus we can construct an injective map $\xi:\suc(u)\to\suc(v)$. If $\Deg_{n-1}(T_2)\le\deg_{n-1}(T_1)$, then $\suc(u)|=|\suc(v)|$ and we can take $\xi$ to be bijective.

For every $u'\in \suc(u)$ use the inductive assumption to find a tower embedding (isomorphism) $\varphi_{u'}:{\downarrow}u'\to{\downarrow}\xi(u')$. The maps  $\varphi_{u'}$, $u'\in\suc(u)$, can be unified to compose a tower embedding $\varphi:{\downarrow}u\to{\downarrow}v$ such that $\varphi(u)=v$ and $\varphi(x)=\varphi_{u'}(x)$ for each $x\in{\downarrow}u'$ with $u'\in\suc(u)$.

If  for some $u_0\in\suc(u)$ and $v_0\in\suc(v)$ we had a tower embedding (isomorphism) $\varphi_0:{\downarrow}u_0\to {\downarrow}v_0$, then we can choose the injection $\xi$ so that $\xi(u_0)=v_0$ and take $\varphi_{u_0}$ be equal to $\varphi_0$.
\end{proof}

Now the proof of Proposition~\ref{p5} becomes easy. Fix any two points $x_1\in[T_1]$ and $y_1\in[T_2]$ and  consider the upper cones ${\uparrow} x_1=\{x_k:k<h(T_1)+1\}$ and ${\uparrow}y_1=\{y_k:k<h(T_2)+1\}$ where $\lev(x_k)=k=\lev(y_k)$ for all $k$. 

Using Lemma~\ref{l0},  construct a sequence of tower embeddings (isomorphisms) $\varphi_n:{\downarrow}x_n\to {\downarrow}y_n$ such that $\varphi_{n+1}|{\downarrow}x_{n}=\varphi_{n}$ for all $n<h(T_1)+1$. Unifying these  embeddings we obtain a desired tower embedding (isomorphism) $\varphi:T_1\to T_2$ defined by $\varphi(x)=\varphi_n(x)$ for $x\in{\downarrow}x_n$.
\end{proof}


We define a tower $T$ to be {\em homogeneous} if $\deg_n(T)=\Deg_n(T)$ for all $n\in \IN$ (and consequently, $\deg^n_k(T)=\Deg^n_k(T)$ for all $k\le n$).

Applying Proposition~\ref{p5a} to homogeneous towers we get

\begin{corollary}\label{p5} For two homogeneous towers $T_1,T_2$ there is a tower isomorphism $\varphi:T_1\to T_2$ if and only if $\deg_k(T_1)=\deg_k(T_2)$ for all $k\in\IN$.
\end{corollary}

A typical example of a homogeneous tower can be constructed as follows. Let $G$ be a group written as the countable union $G=\bigcup_{n\in\IN}H_n$ of an increasing sequence of subgroups $H_n\subset H_{n+1}$. The set $T=\{gH_n:g\in G,\; n\in\IN\}$ is an tower with respect to the inclusion order ($A\le B$ iff $A\subset B$). Observe that the degree of any element $gH_n$ in $T$ is equal to the index of the subgroup $H_{n-1}$ in the group $H_{n}$ (here we assume that the subgroup $H_0$ is trivial).

In particular, for every sequence $\vec k=(k_n)_{n\in\IN}$ of positive integers we can consider the direct sum $G=\oplus_{n\in\IN}\IZ/k_n\IZ$ of cyclic groups and the subgroups $H_n=\oplus_{i<n}\IZ/k_i\IZ$, $n\in\IN$. The corresponding tower $\{gH_n:g\in G,\; n\in\w\}$ will be denoted by $T_{\vec k}$. For this tower we get $\deg_{n}(T_{\vec k})=k_n$ for all $n\in\w$. 
The tower $T_{\vec 2}$ for the constant sequence $k_n=2$, $n\in\w$, will be called the {\em binary tower}.  It is easy to see that the base $[T_{\vec 2}]$ of the binary tower $T_{\vec 2}$ is bijectively asymorphic to the anti-Cantor set $2^{<\w}$. 
\smallskip

The other natural examples of towers appear as canonical $\vec r$-towers of ultra-metric spaces. For each ultra-metric space $X$ and an unbounded increasing sequence $\vec r=(r_n)_{n\in\IN}$ of real numbers the {\em canonic $\vec r$-tower} $T_X(\vec r)$ of $X$ is defined as follows.

For a point $x\in X$ and a real number $r$ by $B_r(x)$ we denote the (closed-and-open) $r$-ball centered at $x$. 
Consider the family $T_X(\vec r)=\{(B_{r_n}(x),n):x\in X,\; n\in\IN\}$ endowed with the partial order $(B_{r_n}(x),n)\le (B_{r_m}(r),m)$ iff $n\le m$ and $B_{r_n}(x)\subset B_{r_m}(x)$. In the following proposition we shall show that $T_X(\vec r)$ is indeed a tower.

\begin{proposition}\label{hom} If $X$ is a (homogeneous) ultra-metric space, then for any unbounded increasing number sequence $\vec r=(r_n)_{n\in\IN}$ the partially ordered set $$T_X(\vec r)=\{(B_{r_n}(x),n):x\in X,\; n\in\IN\}$$ is a (homogeneous) tower whose  base $[T_X(\vec r)]$ is asymorphic to $X$. Moreover, if $r_1=0$, then $[T_X(\vec r)]$ is bijectively asymorphic to $X$.
\end{proposition}

\begin{proof} The proof follows easily from the fact that for any points $x,y$ of $X$ and numbers $r\le R$ the inclusion $B_r(x)\subset B_R(y)$ is equivalent to $B_r(x)\cap B_R(y)\ne\emptyset$. The latter fact holds because the ultrametric of $X$ satisfies the strong triangle inequality. Consequently, for any $n\in\IN$ and points $x,y\in X$ with $B_{r_1}(x)\subset B_{r_n}(y)$ the order interval $[(B_{r_1}(x),1),(B_{r_n}(y),n)]$ contains exactly $n$ elements of the set $T_X$. This shows that the last condition of Definition~\ref{d1} is satisfied. The other conditions also follow from the same observation: $B_r(x)\cap B_R(y)\ne \emptyset$ implies $B_r(x)\subset B_R(y)$.

If the ultrametric space $X$ is homogeneous, then any isometry of $X$ induces a tower isomorphism of the tower $T(\vec r)$. This fact can be used to prove that the tower $T_X(\vec r)$ is homogeneous if so is the space $X$.

If $r_1=0$, then the base of the tower $T_X(\vec r)$ consists of the singletons $B_0(x)=\{x\}$, so we can consider the identity map $\id:X\to[T_X]$ assigning to each $x\in X$ its singleton $B_0(x)$ and notice that this map is a bijective asymorphism. If $r_1>0$, then the asymorphness of $X$ and $[T_X(\vec r)]$ follows from Proposition~\ref{next}.
\end{proof}

It is known (see Theorems 3.1.1 and 3.1.3 in \cite{PZ}) that a metric space $X$ is bijectively asymorphic to an ultra-metric space if and only if $X$ is asymptotically zero-dimensional. The latter means that for every real number $D>0$ there is a $D$-discrete cover $\U$ of $X$ with 
$$\mesh(\U)=\sup_{U\in\U}\diam U<+\infty.$$ The $D$-discreteness of $\U$ means that $\dist(U,V)>D$ for any distinct sets $U,V\in\U$. The following proposition is a ``homogeneous'' version of the mentioned result.

\begin{proposition}\label{p1} Each (homogeneous) asymptotically zero-dimensional metric space $(X,d)$ admits an ultrametric $\rho$ such that the metric spaces $(X,d)$ and $(X,\rho)$ are bijectively asymorphic (and the ultra-metric space $(X,\rho)$ is homogeneous). 
\end{proposition}

\begin{proof} Using the definition of the asymptotic zero-dimensionality of $X$, construct an increasing  sequence $(r_n)_{n\in\IN}$ of positive real numbers such that for every $n\in\IN$ the space $X$ has a $r_n$-discrete cover $\U_n$ with $\mesh\,\U_n<r_{n+1}$.

Define two points $x,y\in X$ to be $r_n$-equivalent if there is a chain of points $x=x_0,x_1,\dots,x_k=y$ in $X$ with $\dist(x_{i-1},x_i)\le r_n$ for all $i\le k$.  It is clear that the $r_n$-equivalence is indeed an equivalence relation, which divides the space $X$ into the equivalence classes. Let $C_x$ denote the equivalence class of a point $x\in X$ and let $\mathcal C_n=\{C_x:x\in X\}$. It is clear that the cover $\C_n$ is $r_n$-discrete and $$\mathcal B(r_n)\prec\mathcal C_n\prec\mathcal B(r_{n+1})$$ where $\mathcal B(r)=\{B_r(x):x\in X\}$ is the cover of $X$ by closed $r$-balls, and for two covers $\U,\V$ of $X$ we write $\U\prec\V$ if each set $U\in\U$ lies in some set $V\in\V$.  

Now define the ultra-metric $\rho$ on $X$ letting $$\rho(x,y)=\max\{n\in\w: \{x,y\}\not\prec\C_n\}$$
for different points $x,y\in X$. It is easy to see that the identity map $(X,d)\to (X,\rho)$ is a bijective asymorphism and  each bijective isometry $f:X\to X$ of the metric space $(X,d)$ is an isometry of the metric space $(X,\rho)$. Consequently, the  ultra-metric space $(X,\rho)$ is homogeneous if so is the space $(X,d)$.
\end{proof}

\section{Admissible morphisms of towers}

Let $T_1,T_2$ be two towers.
A map $\varphi:A\to T_2$ defined on a lower subset $A={\downarrow}A$ of $T_1$ is called an {\em admissible morphism} if
\begin{enumerate}
\item $\lev(\varphi(a))=\lev(a)$ for all $a\in A$;
\item $a\le a'$ in $A$ implies $\varphi(a)\le \varphi(a')$;
\item $\varphi(a)=\varphi(a')$ for $a,a'\in A$ implies that $a,a'\in\suc(v)$ for some $v\in T$;
\item $\varphi(A)$ is a lower subset of $T_2$;
\item $|\varphi(\max A)|\le 1$, 
\end{enumerate}where $\max A$ stands for the (possibly empty) set of maximal elements of the domain $A$.

\begin{lemma}\label{l1} Let $\varphi:T_1\to T_2$ be an admissible morphism between towers $T_1,T_2$. Then the restriction $\Phi=\varphi|[T_1]:[T_1]\to[T_2]$ is an asymorphism.
\end{lemma}

\begin{proof} Given any $n\in\w$ and any subset $A\subset[T_1]$ with $\diam A\le 2n$ we conclude that $A\subset{\downarrow}v$ for some $v\in L_{n+1}$. The monotonicity of $\varphi$ implies that $\varphi(A)\subset\varphi({\downarrow} v)={\downarrow}\varphi(v)$ and thus
$$\diam(\varphi(A))\le\diam({\downarrow}\varphi(v))\le 2n$$ because $\lev(\varphi(v))=\lev(v)=n+1$.

Now assume conversely that $B\subset [T_2]$ is a subset with $\diam(B)\le 2n$. We claim that $\diam(\varphi^{-1}(B))\le 2n+2$. Take any two points $x,y\in\varphi^{-1}(B)$. The inequality $\diam(B)\le 2n$ implies that $B\subset{\downarrow}b$ for some $b\in T_2$ with $\lev(b)=n+1$. 
Let $x',y'\in L_n$ be two points with $x\le x'$ and $y\le y'$. It follows that $\lev(\varphi(x'))=\lev(x')=n+1=\lev(y')=\lev(\varphi(y'))$. 

We claim that $\varphi(x')=b$. For the smallest lower bound $v=\sup(b,\varphi(x'))$,  consider the lower cone ${\downarrow}v$ that contains the point $\varphi(x)$ as a minimal element. Since the order interval $[\varphi(x),v]$ is well-ordered and contains two elements $b$ and $\varphi(x')$ at the same level, we conclude that $\varphi(x')=b$. 
By the same reason $\varphi(y')=b$. Since $\varphi$ is an admissible morphism, the equality $\varphi(x')=\varphi(y')$ implies that $x',y'\in\suc(w)$ for some point $w\in T$. It follows that $\lev(w)=\lev(x')+1=\lev(b)+1=n+2$ and hence $$\dist(x,y)=2\,\lev (\sup(x,y))-2\,\le 2\lev(\sup(x',y'))-2\le 2\,\lev(w)-2=2n+2.$$
\end{proof}

For a real number $r$ denote by 
$$\lfloor r\rfloor=\min\{n\in\IZ:r\le n\}\mbox{ and }\lceil r\rceil=\max\{n\in\IZ:r\ge n\}$$
two nearest integer numbers to $r$.

The following lemma is a crucial step in the proof of Theorem~\ref{ast}.

\begin{lemma}\label{l2} For two towers $T_1,T_2$  there is a surjective admissible morphism $\varphi:T_1\to T_2$ provided there are two sequences $(a_i)_{i\in\IN}$ and $(b_i)_{i\in\IN}$ of reals such that $1\le a_i\le a_i+2\le b_i$, $\lceil a_i\rceil\le \deg_i(T_1)$, and 
$$b_i+a_i\cdot\frac{\Deg_i(T_2)}{a_{i+1}}\le \deg_i(T_1)\le\Deg_i(T_1)\le a_i+b_i\cdot\Big(\frac{\deg_i(T_2)}{b_{i+1}}-2\Big)$$for all $i\in\IN$. 
\end{lemma}

\begin{proof} We define a subset $A\subset T_1$ to be {\em admissible} if $A \subset\suc (v)$ for some $v\in L_k$, $k\in\w$, and $a_k\le|A|\le b_k$. In this case we write $v=\sup(A)$.

Our lemma will be derived from the following
\smallskip

\noindent{\bf Claim 1.} {\em For any admissible subset $A\subset T_1$ and any $w\in T_2$ with $\lev(A)=\lev(w)$ there is an admissible morphism $\varphi:{\downarrow}A\to {\downarrow}w\subset T_2$.
Moreover, if we had an admissible morphism $\varphi_0:{\downarrow}A_0\to {\downarrow}w$ defined on the lower set of an admissible subset $A_0\subset {\downarrow}A$ with $\sup A_0\in A$, then the admissible morphism $\varphi$ can be chosen so that $\varphi|{\downarrow}A_0=\varphi_0$.}
\smallskip

This claim will be proven by induction on the level $\lev(w)$ of the point $w\in T_2$.
If $\lev(w)=1$, then there is noting to construct: just take $\varphi:{\downarrow}A\to\{w\}$ be the constant map. Assume that the claim is proved for all points $w\in T_2$ with $\lev(w)\le n$. 

Take any point $w\in T_2$ with $\lev(w)=n+1$ and let $A\subset T_1$ be an admissible subset with $\lev(A)=\lev(w)=n+1$. For every point $x\in A$ choose a number $d_x\in\{\lfloor \deg(w)/|A|\rfloor,\lceil \deg(w)/|A|\rceil\}$ so that $\sum_{x\in A}d_x=\deg(w)$. 

For every $x\in A$ write the set $\suc(x)$ as a disjoint union $\suc(x)=\cup \A_x$ of a family of admissible sets with cardinality $|\A_x|=d_x$. This is possible because $$
\begin{aligned}
&b_{n}+a_{n}(d_x-1)\le b_{n}+a_{n}\frac{\deg(w)}{|A|}\le b_{n}+a_{n}\frac{\Deg_n(T_2)}{a_{n+1}}\le\\
&\le  \deg_n(T_1)\le\deg(x)\le\Deg_n(T_1)\le a_{n}+b_{n}\cdot\Big(\frac{\deg_n(T_2)}{b_{n+1}}-2\Big)\le\\
&\le a_{n}+b_{n}\cdot\big(\frac{\deg(w)}{|A|}-2)\le a_{n}+b_{n}\cdot(d_x-1).
\end{aligned}$$ 
Moreover, those inequalities guarantee that we can choose the family $\A_x$ to contain an admissible set of any cardinality between $a_{n}$ and $b_{n}$.

Then the family $\A=\bigcup_{x\in A}\A_x$ has cardinality $|\A|=\deg(w)$ and hence we can find a bijective map $f:\A\to\suc(w)$. By the inductive assumption, for each set $A'\in\A$ we can find an admissible surjective homomorphism $\varphi_{A'}:{\downarrow} A'\to{\downarrow} f(A')$. Now define an admissible homomorphism $\varphi:{\downarrow}A\to{\downarrow}w$ letting
$$\varphi(x)=
\begin{cases}
\varphi_{A'}(x)&\mbox{ if $x\in{\downarrow}A'$ for some $A'\in\A$};\\
b&\mbox{ if $x\in A$}.
\end{cases}
$$  

If for some admissible subset $A_0\subset {\downarrow}A$ with $\sup A_0\in A$ we are given with an admissible morphism $\varphi_0:{\downarrow}A_0\to {\downarrow}w$, then we can include the admissible set $A_0$ into the family $\A$ and choose the admissible morphism $\varphi_{A_0}$ equal to $\varphi_0$. This completes the proof of  Claim 1.
\smallskip

To prove the lemma, take increasing sequences $\{x_n:n\in\IN\}\subset T_1$ and $\{y_n:n\in\IN\}\subset T_2$  with $\lev(x_n)=n=\lev(y_n)$ for all $n\in\IN$. For every $n\in\IN$ by induction choose an admissible subset $A_n\subset T_1$ such that $x_n\in A_n\subset\suc(x_{n+1})$. Such a choice is possible because $\lceil a_n\rceil\le \deg_n(T_1)\le \deg(x_{n+1})$. Then ${\downarrow}A_n\subset {\downarrow}A_{n+1}$. Using Claim 1, we can construct a sequence $\varphi_n:{\downarrow} A_n\to {\downarrow}y_n$, $n\in\IN$, of surjective admissible morphisms such that $\varphi_{n+1}|{\downarrow}A_n=\varphi_n$. The union $\varphi=\bigcup_{n\in\IN}\varphi_n:T_1\to T_2$ is a well-defined admissible morphism.
\end{proof}

\section{Asymptotically homogeneous towers}\label{sas}

In this section we shall apply Lemma~\ref{l2} in order to prove that the base $[T]$ of each asymptotically homogeneous tower $T$ is asymorphic to the anti-Cantor set. Let us observe that a tower $T$ is proper (as a metric space) if $[T]$ is unbounded in the path metric of $T$ and the lower set ${\downarrow}x$ of each point $x\in T$ is finite. 

\begin{definition} A tower $T$ is called {\em asymptotically homogeneous} if $T$ is proper and there is a real constant $C$ such that
$$\prod_{k=n}^m\frac{\Deg_k(T)}{\deg_k(T)}\le C$$
 for every $k\le n$. This is equivalent to saying that the infinite product
$$\prod_{k=1}^\infty\frac{\Deg_k(T)}{\deg_k(T)}$$is convergent. 
\end{definition}

The following lemma is a crucial step in the proof of Theorem~\ref{asym} below.

\begin{lemma}\label{l3x} For any asymptotically homogeneous tower $T$ there are real sequences $(a_n)$, $(b_n)$, and increasing number sequences $(n_i)$, and $(m_i)$ such that 
\begin{equation}\label{eq0b}
1\le a_{i}\le a_{i}+2\le b_{i},\quad a_{i}+1\le \deg_{n_{i}}^{n_{i+1}}(T)
\end{equation} and 
\begin{equation}\label{eq1a}
b_{i}+a_{i}\frac{2^{m_{i+1}-m_{i}}}{a_{i+1}}\le \deg^{n_{i+1}}_{n_{i}}(T)\le\Deg^{n_{i+1}}_{n_{i}}(T)\le a_{i}+b_{i}\Big(\frac{2^{m_{i+1}-m_{i}}}{b_{i+1}}-2\Big)
\end{equation}
for all $i\in\IN$.
\end{lemma}

\begin{proof} Those sequences will be constructed by induction. However we should first make some preparatory work. The asymptotic homogeneity of $T$ allows us to find a sequence of real numbers $c_i>1$, $i\in\IN$, such that 
$$\Deg_i(T)\le c_i\cdot \deg_i(T),\;\;i\in\IN,$$and
the infinite product $\prod_{i=1}^\infty c_i$ converges to some real number $C_1^\infty$. Also fix any sequence of real numbers $\delta_i>1$ with convergent infinite product $\prod_{i=1}^\infty\delta_i$. 

For numbers $1\le i\le j\le\infty$ let 
$$C_i^j=\prod_{k=i}^{j-1}c_k\mbox{ and }\delta_i^j=\prod_{k=i}^{j-1}\delta_k.$$To simplify the notation, for $i\le j$ we put $d_i^j=\deg_i^j(T)$ and $D_i^j=\Deg_i^j(T)$. 
It follows from the choice of the numbers $c_i$ that 
\begin{equation}\label{eq0a}
D_i^j\le C_i^j\cdot d_i^j.
\end{equation}

By induction, for every $i\in\IN$ we shall construct real numbers $a_i,b_i$ and positive integers $n_i,m_i$ that satisfy the conditions (\ref{eq1a}) and 
\begin{equation}\label{eq2a}
\frac{b_i}{a_i}\ge C_i^\infty\cdot\delta_i^\infty
\end{equation}

To start the induction, let $n_1=1$, $m_1=0$, and choose any real numbers $a_1,b_1$ satisfying the inequalities $$1\le a_1<a_1+2\le b_1\mbox{ and }b_1\ge a_1\cdot C_1^\infty\cdot\delta_1^\infty.$$

Assume that the numbers $a_{i}$, $b_{i}$, $n_{i}$, $m_{i}$ satisfying (\ref{eq1a}) and (\ref{eq2a}) have been constructed. 

Since the base $[T]$ of $T$ is unbounded, the sequence $(D_{n_i}^n)_{n\ge n_i}$ is unbounded. This fact combined with the almost homogeneity of $T$ implies that the  sequence $(d_{n_{i}}^n)_{n\ge n_{i}}$ is unbounded too. Consequently, there is a number $n_{i+1}>n_{i}$ such that $d_{n_{i}}^{n_{i+1}}>\lceil a_{i}\rceil$ and
\begin{equation}\label{eq3a}
\frac{d_{n_{i}}^{n_{i+1}}-b_{i}}{d_{n_{i}}^{n_{i+1}}+2b_{i}}\ge \frac1{\delta_{n_{i}}}.
\end{equation}
Next, find a number $m_{i+1}>m_{i}$ such that 
\begin{equation}\label{eq4a}
2^{m_{i+1}-m_{i}}(C_{n_{i+1}}^\infty\cdot\delta_{n_{i+1}}^\infty-1)\frac{a_{i}}
{d_{n_{i}}^{n_{i+1}}-b_{i}}>2
\end{equation}and the numbers $a_{i+1}$ and $b_{i+1}$ defined by 
\begin{equation}\label{eq5a}a_{i+1}=\frac{2^{m_{i+1}-m_{i}}a_{i}}{d_{n_{i}}^{n_{i+1}}-b_{i}}\mbox{ and } b_{i+1}=\frac{2^{m_{i+1}-m_{i}}b_{i}}{C_{n_{i}}^{n_{i+1}}
d_{n_{i}}^{n_{i+1}}+2b_{i}-a_{i}}
\end{equation}are greater than 1. 
We claim that the so defined numbers $n_{i+1}$, $m_{i+1}$, $a_{i+1}$, $b_{i+1}$, satisfy the inductive assumptions. In fact, the condition (\ref{eq1a}) follows directly from the definitions of the numbers $a_{i+1}$ and $b_{i+1}$ and the inequality (\ref{eq0a}). To see that (\ref{eq2a}) also holds, observe that
$$\frac{b_{i+1}}{a_{i+1}}=\frac{b_{i}}{a_{i}}\cdot\frac{d_{n_{i}}^{n_{i+1}}-b_{i}}
{C_{n_{i}}^{n_{i+1}}d_{n_{i}}^{n_{i+1}}+2b_{i}-a_{i}}\ge 
\frac{b_{i}}{a_{i}}\cdot \frac1{C_{n_{i}}^{n_{i+1}}}\cdot\frac{d_{n_{i}}^{n_{i+1}}-b_{i}}{d_{n_{i}}^{n_{i+1}}+
2b_{i}}\ge$$and using (\ref{eq3a}), the trivial inequality $\delta_{n_{i}}^{n_{i+1}}\ge \delta_{n_{i}}$, and the inductive assumption $\frac{b_{i}}{a_{i}}\ge C_{n_{i}}^\infty\cdot\delta_{n_{i}}^\infty$ we can continue as
$$\ge \frac{b_{i}}{a_{i}}\cdot\frac1{C_{n_{i}}^{n_{i+1}}}\cdot\frac1{\delta_{n_{i}}}\ge C_{n_{i}}^\infty\cdot\delta_{n_{i}}^\infty\cdot\frac1{C_{n_{i}}^{n_{i+1}}}\cdot\frac1{
\delta_{n_{i}}^{n_{i+1}}}=C_{n_{i+1}}^\infty\cdot \delta_{n_{i+1}}^\infty.$$
The lower bound $\frac{b_{i+1}}{a_{i+1}}\ge C_{n_{i+1}}^\infty\cdot\delta_{n_{i+1}}^\infty>1$ combined with the choice of $m_{i+1}$ in (\ref{eq4a}) yields the condition (\ref{eq0b}):
$$b_{i+1}-a_{i+1}=\Big(\frac{b_{i+1}}{a_{i+1}}-1\Big)\,a_{i+1}\ge (C_{n_{i+1}}^\infty\cdot\delta_{n_{i+1}}^\infty-1)\frac{2^{m_{i+1}-m_{i}}\,a_{i}}
{d_{n_{i}}^{n_{i+1}}-b_{i}}>2.$$
This finishes the inductive step, and also the proof of the lemma.
\end{proof}

We apply Lemmas~\ref{l2} and \ref{l3x} to prove the main result of this section. 

\begin{theorem}\label{asym} The base $[T]$ of each asymptotically homogeneous tower $T$ is asymorphic the anti-Cantor set $2^{<\w}$.
\end{theorem}

\begin{proof} Let $(a_i)_{i=1}^\infty$, $(b_i)_{i=1}^\infty$, $\vec n=(n_i)_{i=1}^\infty$ and $\vec m=(m_i)_{i=1}^\infty$ be the sequences from Lemma~\ref{l3x}. Let $T(\vec n)$ be the level $\vec n$-subtower of $T$. By Proposition~\ref{next}, the map $\next_1:[T]\to[T(\vec n)]$ is an asymorphism. By the same reason, the map $\next_2:[T_2]\to [T_2(\vec m)]$ from the base of the binary tower $T_2$ to the base of its level $\vec m$-subtower $T_2(\vec m)$ is an asymorphism.

 Observe that 
$$
\begin{aligned}
&\deg_i(T_{\vec 2}(\vec m))=2^{m_{i+1}-m_{i}},\\
&\deg_i(T(\vec n))=\deg^{n_{i+1}}_{n_{i}}(T),\\
&\Deg_i(T(\vec n))=\Deg^{n_{i+1}}_{n_{i}}(T)
\end{aligned}
$$ 
which allows us to apply Lemma~\ref{l2} to find an admissible morphism $\varphi:T(\vec n)\to T_{2}(\vec m)$. By Lemma~\ref{l1}, $\varphi$ induces an asymorphism between the bases $[T(\vec n)]$ and $[T_2(\vec m)]$. Finally we obtain an asymorphism  between $[T]$ and the anti-Cantor set $2^{<\w}$ as the composition  of the asymorphisms
$$[T]\sim [T(\vec k)]\sim [T_2(\vec m)]\sim [T_2]\sim 2^{<\w}.$$
\end{proof}

\section{Proof of Theorem~\ref{ast}.} 

We should prove that an unbounded ultra-metric space $X$ of bounded geometry is asymorphic to the anti-Canor set provided there is an increasing unbounded sequence $\vec r=(r_n)_{n\in\IN}$ such that
\begin{equation}\label{eq:ast}\prod_{n\in\IN}\frac{\Ent_{r_n}^{r_{n+1}}(X)}{\ent_{r_n}^{r_{n+1}}(X)}<\infty.
\end{equation}
By Proposition~\ref{hom}, $X$ is asymorphic to the base $[T_X(\vec r)]$ of 
the canonic $\vec r$-tower $T_X(\vec r)=\{(B_{r_n}(x),n):x\in X,\; n\in\IN\}$. The entropy condition (\ref{eq:ast})  is equivalent to the asymptotic homogeneity of the tower $T_X(\vec r)$. Applying Theorem~\ref{asym}, we conclude that the anti-Cantor set $2^{<\w}$ is asymorphic to the base $[T_X(\vec r)]$ of $T_X(\vec r)$ and hence is also asymorphic to $X$.

\section{Proof of Proposition~\ref{ent}.}

1. Assume that a metric space $X$ is coarsely equivalent to a subspace $Z$ of a metric space $Y$. By Proposition~\ref{ascors}, there is an asymorphism $\Phi:X\Ra Z\subset Y$, which is an asymorphic embedding of $X$ into $Y$.

Find $\e>0$ such that $\Ent_\e^\sharp(Y)=\Ent^\sharp(Y)$. Since $\Phi^{-1}:Y\Ra X$ is bornologous, there is $\e'>0$ such that $\diam(\Phi^{-1}(B))\le\e'$ for every bounded subset $B\subset Y$ with $\diam(B)\le 2\e$. 

We claim that $\Ent^\sharp_{\e'}(X)\le\Ent_\e^\sharp(Y)$. This inequality will follow as soon as we check that $\Ent_{\e'}^\delta(X)<\Ent_\e^\sharp(Y)$ for every $\delta<\infty$. Since the multi-map $\Phi$ is bornologous, there is a real number $\delta'$ such that $\diam(\Phi(A))\le\delta'$ for any bounded subset $A\subset X$ with $\diam(A)\le 2\delta$.  We claim that
\begin{equation}\label{eqp1}
\Ent_{\e'}^\delta(X)\le \Ent_{\e}^{\delta'}(Y)<\Ent_{\e}^\sharp(Y)=\Ent^\sharp(Y).
\end{equation} The strict inequality follows from the definition of $\Ent_\e^\sharp(Y)$. To prove the other inequality, take any $x_0\in X$ and observe that $\diam\big(\Phi(B_\delta(x_0))\big)\le\delta'$ and thus $\Phi(B_\delta(x_0))\subset B_{\delta'}(y_0)$ for some $y_0\in Y$.

It follows that the ball $B_{\delta'}(y_0)$ has an $\e$-net $N\subset  B_{\delta'}(y_0)$ of size $|N|\le\Ent_\e^{\delta'}(Y)$. Consider the subset $N_1=\{y\in N:\dist\big(y,\Phi(B_{\delta}(x_0))\big)<\e\}$ and for every $y\in N_1$ choose a point $y'\in \Phi(B_{\delta}(x_0))$ with $\dist(y',y)<\e$. Then the set $N_2=\{y':y\in N_1\}$ is a $2\e$-net for $\Phi(B_\delta(x_0))$ of size $|N_2|\le|N_1|\le|N|$.

For every $y\in N_2$ pick a point $\xi(y)\in\Phi^{-1}(y)\cap B_\delta(x_0)$. We claim that the set $M=\{\xi(y):y\in N_2\}$ is an ${\e'}$-net for $B_\delta(x_0)$. Indeed, for every $a\in  B_\delta(x_0)$ and every $b\in\Phi(a)$ we can find a point $y\in N_2$ with $\dist(b,y)<2\e$. Observe that $\{a,\xi(y)\}\subset\Phi^{-1}(\{b,y\})$.
Since $\diam(\{y,b\})<2\e$, the choice of ${\e'}$ guarantees that $$\diam\,\{a,\xi(y)\}\le\diam\,\Phi^{-1}(\{b,y\})\le {\e'}$$witnessing that $M$ is an ${\e'}$-net for $B_\delta(x_0)$. Therefore, $\Ent_{\e'}(B_\delta(x_0))\le|M|\le|N_2|\le|N|\le\Ent_\e^{\delta'}(Y)$ and (\ref{eqp1}) holds.
Now we see that
$$\Ent^\sharp(X)\le\Ent^\sharp_{\e'}(X)\le\Ent^\sharp_\e(Y)=\Ent^\sharp(Y).$$
\smallskip

2. Assume that two metric spaces $X,Y$ are asymorphic and let $\Phi:X\Ra Y$ be an asymorphism. It follows from the preceding case that $\Ent^\sharp(X)=\Ent^\sharp(Y)$. 

Now we shall prove that $\ent^\sharp(X)\le\ent^\sharp(Y)$.
Find $\e>0$ with $\ent_\e^\sharp(Y)=\ent^\sharp(Y)$. The bornologity of $\Phi^{-1}$ yields us a real number $\e'$ such that $\diam(\Phi^{-1}(B))<\e'$ for any subset $B\subset Y$ of diameter $\diam(B)\le 2\e$. We claim that $$\ent^\sharp(X)\le\ent^\sharp_{\e'}(X)\le\ent^\sharp_\e(Y)=\ent^\sharp(Y).$$
Assuming conversely that 
$\ent_{\e'}^\sharp(X)>\ent_\e^\sharp(Y)$, we could find $\delta<\infty$ such that $(\ent_{\e'}^\delta(X))^+>\ent_{\e}^\sharp(Y)$, which is equivalent to $\ent_{\e'}^\delta(X)\ge\ent^\sharp(Y)$. 
The bornologity of $\Phi$ yields a real number $\delta'$ such that $\diam \Phi(A)\le\delta'$ for any subset $A\subset X$ with $\diam(A)\le2\delta$. The definition of $\ent_\e^\sharp(Y)$ implies that $\min_{y\in Y}\Ent_\e(B_{\delta'}(y))=\ent_\e^{\delta'}(Y)<\ent_\e^\sharp(Y)$ and thus there is a point $y_0\in Y$ with $\Ent_\e(B_{\delta'}(y_0))<\ent_\e^\sharp(Y)$. This means that the ball $B_{\delta'}(y_0)$ contains an $\e$-net $N$ of size $|N|<\ent_\e^\sharp(Y)$.

Now take any point $x_0\in\Phi^{-1}(y_0)$ and consider the closed $\delta$-ball $B_\delta(x_0)\subset X$. It follows from the choice of $\delta'$ that $\diam \Phi(B_\delta(x_0))\le\delta'$. Since $y_0\in\Phi(x_0)\subset\Phi(B_\delta(x_0))$, we conclude that $\Phi(B_\delta(x_0))\subset B_{\delta'}(y_0)$. Repeating the argument from the preceding item, we can transform the $\e$-net $N$ into an $\e'$-net $M\subset B_\delta(x_0)$ of cardinality $|M|\le|N|$. Then 
$$\ent_{\e'}^\delta(X)\le\Ent_{\e'}(B_\delta(x_0))\le|M|\le|N|<\ent_\e^\sharp(Y)$$ which is a desired contradiction that proves the inequality $\ent^\sharp(X)\le\ent^\sharp(Y)$.

The reverse inequality $\ent^\sharp(X)\ge\ent^\sharp(Y)$ can be proved by analogy.
\smallskip

3. Assume that $X,Y$ are two ultra-metric spaces with $\Ent^\sharp(X)\le\ent^\sharp(Y)$. Find a real number $R$  such that $\Ent_{r}^\sharp(X)=\Ent^\sharp(X)$ and $\ent_{r}^\sharp(Y)=\ent^\sharp(Y)$ for all $r\ge R$. Using the definition of $\Ent^\sharp_{r}(X)$ we can find two unbounded increasing sequences of real numbers $\vec r=(r_n)_{n\in\IN}$ and $\vec \rho=(\rho_n)_{n\in\IN}$ such that $r_1=R=\rho_1$ and $\Ent_{r_{n+1}}^{r_n}(X)\le\ent_{\rho_{n+1}}^{\rho_n}(Y)$ for all $n\in\IN$.

It follows from Proposition~\ref{hom} that $X$ is asymorphic to the base $[T_X(\vec r)]$ of the canonic level $\vec r$-subtower $T_X(\vec r)$ while $Y$ is asymorphic to the base $[T_Y(\vec \rho)]$ of the level $\vec \rho$-subtower $T_Y(\vec\rho)$ of $Y$. 
Let $\Phi_X:X\to [T_X(\vec r)]$ and $\Phi_Y:Y\to [T_Y(\vec \rho)]$ be the corresponding asymorphisms.

Observe that $\Deg_i^j(T_X(\vec r))=\Ent_{r_i}^{r_j}(X)$ and $\deg_i^j(T_X(\vec r))=\ent_{r_i}^{r_j}(X)$ for all $i<j$. This implies that $\Deg_n(T_X(\vec r))\le \deg_n(T_Y(\vec\rho))$ for all $n\in\IN$. Applying Proposition~\ref{p5a} we can find a tower embedding $\varphi:T_X(\vec r)\to T_Y(\vec\rho)$ which induces an isometric embedding $\psi=\varphi|[T_X(\vec r)]:[T_X(\vec r)]\to [T_Y(\vec \rho)]$. Now we see that the multi-map $\Psi=\Phi_Y^{-1}\circ\psi\circ \Phi_X:X\Ra Y$ is an asymorphic embedding. Considered as a multi-map into $\Psi(X)\subset Y$, $\Psi:X\to\Psi(X)$ is an asymorphism of $X$ onto the subspace $\Psi(X)$ of $Y$. By Proposition~\ref{ascors}, $X$ is coarsely equivalent to $\Psi(X)$. 
\smallskip

3. Let $X$ be a metric space. We need to check that if $\Ent(X)$ is a limit cardinal, then it has countable cofinality. Find a real number $\e>0$ with $\Ent(X)=\Ent_\e(X)$ and notice that $\Ent_\e(X)=\sup_{n\in\IN}(\Ent_\e^n(X))^+$.

Now assume that $\kappa$ is a cardinal $\kappa$ such that either $\kappa=2$ or $\kappa$ is an infinite successor cardinal or else $\kappa$ is a limit cardinal of countable cofinality. We need to find a homogeneous ultra-metric space $X$ with $\Ent^\sharp(X)=\kappa$. For this we consider 3 cases.
\smallskip

(a) If $\kappa\le\aleph_0$, then we have the necessary examples because  $\Ent^\sharp(\{0\})=2$, and $\Ent^\sharp(2^{<\w})=\aleph_0$.
\smallskip

(b) If $\kappa=\lambda^+$ is an infinite successor cardinal, then we can consider the ultra-metric space $\lambda^{<\w}$ and observe $\Ent^\sharp(\lambda^{<\w})=\lambda^+=\kappa$.
\smallskip

(c) Finally assume that $\kappa$ is an uncountable limit cardinal of countable cofinality and choose an increasing sequence of infinite cardinals $(\kappa)_{n\in\IN}$ with $\sup_{n\in\IN}\kappa_n=\kappa$. Let $X=\IQ(\kappa)$ be a linear space over the field $\IQ$ having the set of ordinals  $\kappa=\{\alpha:\alpha<\kappa\}$ for a Hamel basis. For every $n\in\IN$ let $L_n=\IQ(\kappa_n)$ be the linear subspace algebraically generated by the subset $\kappa_n\subset \kappa$. On the space $\IQ(\kappa)$ consider the ultra-metric $$d(x,y)=2\cdot\max\{n\in\IN:x-y\notin L_n\}$$ where $x,y\in X$ are two distinct points of $X$. 

Observe that for every $n<m$ we get $\Ent_{n}(L_m)=|L_m/L_n|=\kappa_m$ and hence $\Ent^\sharp_n(X)=\sup_{m\in\IN}\kappa_m^+=\kappa$ and $\Ent^\sharp(X)=\min_{n\in\IN}\Ent^\sharp_n(X)=\kappa$.

\section{Proof of Theorem~\ref{class}.}\label{s5}

 We need to prove that two homogeneous ultra-metric spaces $X$ and $Y$ are asymorphic if and only  if $\Ent^\sharp(X)=\Ent^\sharp(Y)$. The ``only if'' part follows from 
Proposition~\ref{ent}(2).
\smallskip

To prove the ``if'' part, assume that  $\Ent^\sharp(X)=\Ent^\sharp(Y)=\kappa$.
\smallskip

1. If $\kappa\le 1$, then the metric spaces $X,Y$ are bounded and hence asymorphic.
\smallskip

2. If $\kappa=\aleph_0$, then the spaces $X,Y$, being  homogeneous, are asymorphic to the anti-Cantor set $2^{<\w}$ according to Theorem~\ref{ast}.
\smallskip

3. Assume that $\kappa=\mu^+$ is an infinite successor cardinal. Then we can choose an unbounded increasing sequence $\vec r=(r_n)_{n\in\IN}$ of real numbers such that 
$$\Ent_{r_n}^{r_{n+1}}(X)=\mu=\Ent_{r_n}^{r_{n+1}}(Y)$$ for all $n\in\IN$.
By  Proposition~\ref{hom}, $X$ is asymorphic to the base $[T_X(\vec r)]$ of the (homogeneous) canonic $\vec r$-tower $T_X(\vec r)$ of $X$. 

The same is true for the space $Y$: it is asymorphic to the base $[T_Y(\vec r)]$ of its canonic $\vec r$-tower $T_Y(\vec r)$. By Corollary~\ref{p5}, the homogeneous towers $T_X(\vec r)$ and $T_Y(\vec r)$ are isomorphic, which implies that their bases $[T_X(\vec r)]$ and $[T_Y(\vec r)]$ are isometric. Combining the asymorphisms $$X\sim [T_X(\vec r)]\sim[T_Y(\vec r)]\sim Y$$ we conclude that the spaces $X,Y$ are asymorphic.
\smallskip

4. Finally assume that $\kappa=\Ent^\sharp(X)=\Ent^\sharp(Y)$ is an uncountable limit cardinal. We can choose an unbounded increasing sequence $\vec r=(r_n)_{n\in\IN}$ of real numbers such that the sequences $\kappa_n=\deg^n(T_X(\vec r))$ and $\mu_n=\deg^n(T_Y(\vec r))$, $n\in\IN$, consists of infinite cardinals, are increasing and have $\sup_{n\in\IN}\kappa_n=\kappa=\sup_{n\in\IN}\mu_n$. 

In the item 3(c) of the proof of Theorem~\ref{class} we defined the space $\IQ(\kappa)$ endowed with the ultrametric
$$d_1(x,y)=2\cdot\max\{n\in\IN:x-y\notin\IQ(\kappa_n)\}$$where $x,y\in\IQ(\kappa)$ are distinct points of $\IQ(\kappa)$. This space is isometric to the base of the homogeneous tower $T_1=\{x+\IQ(\kappa_n):x\in\IQ(\kappa_n),\;n\in\IN\}$ with $\deg^n(T_1)=|\IQ(\kappa_{n+1})/\IQ(\kappa_{n})|=\kappa_n $ for all $k\in\IN$ (here we assume that $\kappa_{0}=0$). By Corollary~\ref{p5}, the homogeneous towers $T_X(\vec r)$ and $T_1$ are isomorphic and consequently, their bases $[T_X(\vec r)]$ and $\IQ(\kappa)=[T_1]$ are isometric. Taking into account that $X$ is asymorphic to $[T_X(\vec r)]$, we see that the spaces $X$ and $(\IQ(\kappa),d_1)$ are asymorphic.

By the same reason, $Y$ is asymorphic to  the space $\IQ(\kappa)$ endowed with the ultra-metric $$d_2(x,y)=2\cdot\max\{n\in\IN:x-y\notin\IQ(\mu_n)\}$$where $x,y\in\IQ(\kappa)$ are distinct points of $\IQ(\kappa)$. 

Since the sequences $(\kappa_n)$ and $(\mu_n)$ are strictly increasing and have the same supremum, the identity map $(\IQ(\kappa),d_1)\to (\IQ(\kappa),d_2)$ is a bijective asymorphism. Combining the (bijective) asymorphisms
$$X\sim [T_X(\vec r)]\sim (\IQ(\kappa),d_1)\sim(\IQ(\kappa),d_2)\sim[T_Y(\vec r)]\sim Y$$we conclude that $X$ and $Y$ are asymorphic.

\section{Some Open Problems}

In this paper we characterized homogeneous ultra-metric spaces asymorphic to the anti-Cantor set: those are exactly homogenous ultra-metric spaces with $\Ent^\sharp(X)=\aleph_0$. However for arbitrary (not necessarily homogeneous) metric spaces a similar characterization problem seems to be much more difficult.

\begin{problem} Find necessary and sufficient conditions on an ultra-metric space $X$ guaranteeing that $X$ is asymorphic to the anti-Cantor set $2^{<\w}$. In particular, is $X$ asymorphic to $2^{<\w}$ if $\ent^\sharp(X)=\Ent^\sharp(X)=\aleph_0$?
\end{problem}

We can pose a simpler question asking if the condition in Theorem~\ref{ast} involving infinite products can be replaced by a weaker condition.

\begin{problem} Is a proper ultra-metric space $X$ asymorphic to the anti-Cantor set   if there is a real constant $C$ and an increasing number sequence $(r_i)$ such that
$$\Ent_{r_i}^{r_j}(X)\le C\cdot\ent_{r_i}^{r_j}(x)$$ for all $i<j$?
\end{problem}

This problem is equivalent to the following one.

\begin{problem} Is the base $[T]$ of a proper tower $T$ asymorphic to the anti-Cantor set if 
$$\sup_{i<j}\frac{\Deg_i^j(T)}{\deg_i^j(T)}<\infty?$$
\end{problem}

If the two latter problems have affirmative answers then the following problem concerning the hyperspace $\exp_{\le n}(2^\w)$ of the anti-Cantor set also has an affirmative answer. For a metric space $X$ by $\exp_{\le n}(X)$ we denote the space of non-empty at most $n$-element subsets of $X$ endowed with the Hausdorff metric
$$\mathrm{dist}_H(A,B)=\max\{\max_{a\in A}\dist(a,B),\max_{b\in B}\dist(b,A)\}\mbox{ for $A,B\in\exp_{\le n}(X)$}.$$ The asymptotic properties of the hyperspaces $\exp_{\le n}(X)$ have been studied by O.Shukel in \cite{S}.

\begin{problem} Is the hyperspace $\exp_{\le n}(2^{<\w})$ asymorphic to $2^{<\w}$ for every $n\in\IN$?
\end{problem}

Proposition~\ref{ent}(2) guarantees that each metric space $X$, asymorphic to the anti-Cantor set $2^{<\w}$, has small sharp entropy $\ent^\sharp(X)=\aleph_0$. The simplest unbounded metric space $X$ with $\ent^\sharp(X)=2$ is the quickly increasing number sequence $S=\{n^2:n\in\IN\}$. It is easy to check that $2=\ent^\sharp(S)<\Ent^\sharp(X)=\aleph_0$.

\begin{problem} Characterize ultra-metric spaces $X$ whose product $X\times 2^{<\w}$ with the anti-Cantor set $2^{<w}$ is asymorphic to $2^{<\w}$. In particular, is $S\times 2^{<\w}$ asymorphic to $2^{<\w}$.
\end{problem} 

Here we endow the product  $X\times Y$ of two (ultra-)metric spaces $X,Y$ with the (ultra-)metric
$$\dist\big((x,y),(x',y')\big)=\max\{\dist(x,x'),\dist(y,y')\big).$$

Proposition~\ref{ent}(3) guarantees that an ultra-metric space $X$ contains a coarse copy of the anti-Cantor set $2^{<\w}$ provided $\ent^\sharp(X)=\aleph_0$. However there are ultra-metric spaces $X$ with $\ent^\sharp(X)=2$ containing an isometric copy of $2^{<\w}$. The simplest example of such a space is the subspace
$S\times\{\vec 0\}\cup \{1\}\times 2^{<\w}$ of the product $S\times 2^{<\w}$.

\begin{problem} Characterize metric spaces $X$ that admit an asymorphic embedding $\Phi:2^{<\w}\Ra X$ of the anti-Cantor set.
\end{problem}

\end{document}